\documentclass[12pt]{amsart}
\usepackage{amssymb,amsmath,amsthm}
\usepackage{srcltx}

\newtheorem{proposition}{Proposition}[section]
\newtheorem{theorem}[proposition]{Theorem}
\newtheorem{corollary}[proposition]{Corollary}
\newtheorem{lemma}[proposition]{Lemma}
\theoremstyle{definition}
\newtheorem{definition}[proposition]{Definition}
\newtheorem{remark}[proposition]{Remark}

\numberwithin{equation}{section}

\def\s{{r}}

\def\e{{\rm e}}
\def\eps{\varepsilon}
\def\d{{\rm d}}
\def\dist{{\rm dist}}
\def\ddt{\frac{\d}{\d t}}
\def\R {\mathbb{R}}

\def\A {\mathbb{A}}
\def\HH {{\mathcal H}}
\def\H {{\rm H}}

\def\LL {{\mathfrak L}}
\def\C {{\mathcal C}}
\def\BB {{\mathbb B}}
\def\D {{\rm Dom}}
\def\B {{\mathcal B}}

\def\A {{\mathfrak A}}
\def\K {{\mathcal K}}

\def\E {{\mathcal E}}

\def\Q {{\mathcal Q}}
\def\M {{\mathcal M}}

\def\cbt {{\sc cbt}}
\def \l {\langle}
\def \r {\rangle}
\def \pt {\partial_t}

\def \and{\qquad\text{and}\qquad}

\def \x {\boldsymbol{x}}
\def \au {\rm}
\def \ti {\it}
\def \jou {\rm}
\def \bk {\it}
\def \no#1#2#3 {{\bf #1} (#3), #2.}
\def \eds#1#2#3 {#1, #2, #3.}

\title[Nonlinear viscoelastic equations with memory]
{Global attractors for nonlinear\\
viscoelastic equations with memory}
\author[M. Conti, E.M. Marchini, V. Pata]
{Monica Conti, Elsa M. Marchini and Vittorino Pata}
\address{Politecnico di Milano - Dipartimento di Matematica ``F.Brioschi''
\newline\indent
Via Bonardi 9,  20133 Milano, Italy}
\email{monica.conti@polimi.it}
\email{elsa.marchini@polimi.it}
\email{vittorino.pata@polimi.it}


\begin{document}

\begin{abstract}
We study the asymptotic properties of the semigroup $S(t)$ arising from the
nonlinear viscoelastic equation with hereditary memory
on a bounded three-dimensional domain
$$|\partial_t u|^\rho \partial_{tt} u-\Delta \partial_{tt} u
-\Delta \partial_t u-\Big(1+\int_0^\infty \mu(s)\,\d s \Big)\Delta u+\int_0^\infty \mu(s)\Delta u(t-s)\,\d s +f(u)=h$$
written in the past history framework of Dafermos~\cite{DAF}.
We establish the existence of the global attractor of optimal regularity for $S(t)$
when $\rho\in[0,4)$
and $f$ has polynomial growth of (at most) critical order $5$.
\end{abstract}

\maketitle

\section{Introduction}

\subsection{The model system}
Given a bounded domain $\Omega\subset\R^3$ with smooth boundary $\partial\Omega$
and $\rho\in[0,4]$, we
consider for $t>0$ the system of equations
\begin{equation}
\label{SYS}
\begin{cases}
\displaystyle
|\partial_t u|^\rho\partial_{tt} u-\Delta\partial_{tt}u-\Delta\partial_t u
-\Delta u -\int_0^\infty \mu(s) \Delta\eta(s)\,\d s + f(u)=h\\
\partial_t\eta=-\partial_s\eta+ \partial_t u
\end{cases}
\end{equation}
in the real-valued unknowns
$$u=u(\x,t)\and \eta=\eta^t(\x,s),$$
where $\x\in\Omega$, $t\in[0,\infty)$ and $s\in\R^+=(0,\infty)$.
System~\eqref{SYS} is complemented by the Dirichlet boundary condition
\begin{equation}
\label{SYSBC1}
u(\x,t)_{|\x\in\partial\Omega}=0,
\end{equation}
and by the ``boundary condition" for $\eta$
\begin{equation}
\label{SYSBC2}
\lim_{s\to 0}\eta^t(\x,s)=0.
\end{equation}
The model is subject to
the initial conditions (the dependence on $\x$ is omitted)
$$
u(0)=u_0,\qquad \partial_t u(0)=v_0,
\qquad
\eta^0=\eta_0,
$$
where $u_0, v_0:\Omega\to\R$ and $\eta_0:\Omega\times\R^+\to\R$
are prescribed functions.
The external force $h$ is time-independent,
while the locally Lipschitz nonlinearity $f$, with $f(0)=0$,
fulfills the critical
growth restriction
\begin{equation}
\label{GROW}
|f(u)-f(v)|\leq c|u-v|(1+|u|^4+|v|^4),
\end{equation}
along with the dissipation conditions\footnote{Conditions \eqref{DISS1}-\eqref{DISS2} follow for instance by
requiring $f\in\C^1(\R)$ with
$\displaystyle\liminf_{|u|\to\infty}f'(u)>-\lambda_1$.}
\begin{align}
\label{DISS1}
f(u)u &\geq F(u)-\frac{\lambda_1}2 (1-\nu)|u|^2-m_f\\
\label{DISS2}
F(u)&\geq -\frac{\lambda_1}2 (1-\nu) |u|^2-m_f
\end{align}
for some $\nu>0$ and $m_f\geq 0$.
Here $\lambda_1>0$ denotes the first eigenvalue of the Dirichlet operator $-\Delta$
and
$$F(u)=\int_0^u f(y)\,\d y.$$
Finally,
the convolution (or memory) kernel $\mu$ is a
nonnegative, nonincreasing, piecewise absolutely continuous
function on $\R^+$ of finite total mass
$$\int_0^\infty\mu(s)\,\d s=\kappa\geq 0$$
complying with the further assumption
\begin{equation}
\label{NEC1}
\int_s^\infty\mu(\sigma)\,\d\sigma\leq\Theta\mu(s),
\end{equation}
for some $\Theta>0$.
In particular, $\mu$ is allowed to exhibit (even infinitely many) jumps,
and can be unbounded about the origin.
At the same time, $\mu$ can be identically zero,
yielding the equation
$$|\partial_t u|^\rho \partial_{tt} u-\Delta \partial_{tt} u
-\Delta\partial_t u -\Delta u+f(u)=h.$$

\begin{remark}
Assuming the past history of $u$ to be known, from the second equation of~\eqref{SYS}
together with~\eqref{SYSBC2}
one deduces the formal equality (see \cite{DAF})
\begin{equation}
\label{ForE}
\eta^t(s)=u(t)-u(t-s).
\end{equation}
Accordingly, the first equation becomes
$$|\partial_t u|^\rho\partial_{tt} u-\Delta\partial_{tt}u-\Delta\partial_t u
-(1+\kappa)\Delta u +\int_0^\infty \mu(s) \Delta u(t-s)\,\d s + f(u)=h.$$
This provides a generalization, accounting for memory effects in the material,
of equations of the form
$$\varrho(\partial_t u)\partial_{tt}u-\Delta \partial_{tt} u-\Delta \partial_t u-\Delta u+ f(u)=h,$$
arising
in the description of the vibrations of thin rods whose
density $\varrho$ depends on the velocity $\partial_t u$ (see e.g.\ \cite{Love}).
\end{remark}

\subsection{Earlier contributions}
The Volterra version of~\eqref{SYS} with $f=h\equiv 0$
$$|\partial_t u|^\rho\partial_{tt} u-\Delta\partial_{tt}u-\Delta\partial_t u
-(1+\kappa)\Delta u +\int_0^t \mu(s) \Delta u(t-s)\,\d s=0,$$
corresponding to the choice of the initial datum $\eta_0=u_0$,
has been considered by several authors, also with different kind of damping terms,
in concern with the decay pattern of solutions
(see \cite{CCF,HW1,HW2,Li1,Li2,MM,MT1,MT2,MT3,PP,Wu}).
On the contrary, the asymptotic analysis
of the whole system~\eqref{SYS} has been tackled only in the recent work~\cite{AMQ},
where the authors prove the existence of the global attractor (without any additional regularity)
within the following set of
hypotheses:

\begin{itemize}
\item The nonlinearity $f$ has at most polynomial growth $3$ and
\begin{equation}
\label{kefe}
f(u)u\geq F(u)\geq 0.
\end{equation}
\item The nonnegative kernel $\mu\in \C^1(\R^+)\cap L^1(\R^+)$ fulfills
for some $\delta>0$ and every $s\in\R^+$
the relation
\begin{equation}
\label{palta}
\mu'(s)+\delta \mu(s)\leq 0.
\end{equation}
\item The parameter $\rho\in(1,2]$.
\end{itemize}

\smallskip
In addition, the exponential decay
of solutions is obtained when $h\equiv 0$ (meaning that the global attractor $\A=\{0\}$ is exponential
as well).

\begin{remark}
Actually, analogously to what done in all the other papers on the Volterra case,
the (exponential) decay rate turns out to depend
on the size of the initial data. As we will see
in the next Section~\ref{SecABS}, a simple argument allows to get rid of
such a dependence.
\end{remark}

The restriction on $\rho$ is also motivated by the fact that
the well-posedness result for \eqref{SYS}, hence the existence of the semigroup, was
available only for $\rho\in(1,2]$ (besides for the much simpler case $\rho=0$).
On the other hand, after \cite{CMP} now we know that \eqref{SYS} generates a strongly
continuous semigroup in our more general assumptions, and in particular,
for all $\rho\in[0,4]$. This, of course, opens a new scenario which is worth to
be investigated.

\subsection{The result}
In this work, we prove that the strongly continuous semigroup $S(t)$
generated by system~\eqref{SYS} is dissipative (i.e.\ possesses bounded absorbing sets)
for all $\rho\in[0,4]$. In particular, the exponential decay of solutions occurs
whenever $m_f=0$ and $h\equiv 0$.
Besides, we establish the following theorem.

\begin{theorem}
\label{MAIN}
Let $h\in L^2(\Omega)$ and $\rho\in[0,4)$. Then $S(t)$ possesses the global attractor of optimal regularity.
\end{theorem}

With respect to the earlier literature,
Theorem~\ref{MAIN} improves the picture in several directions:

\begin{itemize}
\item The attractor is bounded in a more regular (in fact, the best possible) space.
\smallskip
\item The nonlinearity $f$ is allowed to reach the critical polynomial order 5, under the very general
dissipation conditions \eqref{DISS1}-\eqref{DISS2}, which include for instance terms of the form
$$f(u)=u^5+au^4+bu^3+cu^2+du+e,$$
not covered by \eqref{kefe}.
\smallskip
\item Condition \eqref{NEC1} on the memory kernel $\mu$ is the most general possible one
(among the class of nonincreasing summable kernels), since its failure prevents the uniform decay
of solutions to systems with memory, no matter how the equations involved are (see~\cite{CHEP}).
As shown in~\cite{RMPS}, the condition can be equivalently stated
as
\begin{equation}
\label{nece}
\mu(\sigma+s)\leq C \e^{-\delta \sigma}\mu(s)
\end{equation}
for some $C\geq1$, $\delta>0$, every $\sigma\geq 0$ and almost every $s>0$.
Observe that the latter inequality with $C=1$ boils down to~\eqref{palta} (actually, for a.e.\ $s\in\R^+$).
At the same time, when $C>1$ a much wider class of memory kernels is admissible
(see~\cite{CHEP} for more comments).
\smallskip
\item The parameter $\rho$ belongs to the interval $[0,4)$.
Nonetheless, the existence of the global attractor in the case $\rho=4$,
which is critical for the Sobolev embedding,
seems to be out of reach at the moment.
\end{itemize}

\subsection*{Plan of the paper}
After introducing the functional setting (Section~\ref{SecFS}),
we dwell on the existence of the solution semigroup
(Section~\ref{SecGS}), whose dissipative features are discussed in Section~\ref{SecABS}.
Our main results on the existence of the global attractor of optimal regularity
are presented in Section~\ref{SecMS}. The remaining three Sections~\ref{SecEGA}-\ref{SecORA}
are devoted to the proofs. The final Appendix contains some technical lemmas.

\section{Functional Setting}
\label{SecFS}

\noindent
We denote by $A=-\Delta$
the Dirichlet operator on $L^2(\Omega)$ with
domain $H^2(\Omega)\cap H_0^1(\Omega)$.
For $r\in\R$,
we define the scale of compactly nested Hilbert spaces
$$\H^r=\D(A^{\frac{r}2}),\qquad
\langle u,v\rangle_r=\langle A^{\frac{r}2} u,A^{\frac{r}2} v\rangle_{L^2(\Omega)},
\qquad
\|u\|_r=\|A^{\frac{r}2} u\|_{L^2(\Omega)}.$$
The index $r$ is omitted whenever zero. In particular,
$$\H^{-1}=H^{-1}(\Omega),\qquad \H=L^2(\Omega),\qquad
\H^{1}=H_{0}^{1}(\Omega),
\qquad \H^2=H^2(\Omega)\cap H_0^1(\Omega),$$
and we have the generalized Poincar\'e inequalities
$${\lambda_1}\|u\|_r^2\leq \|u\|_{1+r}^2.$$

\begin{remark}
It is readily seen that the dissipation conditions \eqref{DISS1}-\eqref{DISS2} imply
\begin{align}
\label{extra1}
&\l f(u),u\r\geq \l F(u),1\r-\frac12(1-\nu)\|u\|_1^2-M_f,\\
\label{extra2}
&\l F(u),1\r\geq -\frac12(1-\nu)\|u\|_1^2-M_f,
\end{align}
where $M_f=m_f|\Omega|$.
\end{remark}

Next, we introduce the history spaces
$$
\M^r=L^2_\mu(\R^+;\H^{1+r})
$$
endowed with the inner products
$$\langle \eta,\xi\rangle_{\M^r}=
\int_0^{\infty}\mu(s)\langle\eta(s),\xi(s)\rangle_{1+r}\,\d s.
$$
We will also consider the infinitesimal generator $T$
of the right-translation semigroup on $\M$
defined as
$$T\eta=-\eta',\qquad
\D(T)=\big\{\eta\in{\M}:\eta^\prime\in\M,\,\,
\eta(0)=0\big\},$$
the {\it prime}
standing for weak derivative. The following inequality holds (see e.g.\ \cite{GP})
\begin{equation}
\label{TTT}
\l T\eta,\eta\r_{\M}\leq 0,\quad\forall\eta\in\D(T).
\end{equation}
Finally, we introduce the extended history spaces
$$
\HH^r=\H^{1+r}\times\H^{1+r}\times \M^r.$$

\subsection*{Notation}
Throughout the paper, $c\geq 0$ and $\Q(\cdot)$ will stand for a generic constant and
a generic increasing positive function, respectively.
We will use, often without explicit mention,
the usual Sobolev embeddings, as well as the Young, H\"older and Poincar\'e inequalities.

\section{The Gradient System}
\label{SecGS}

\noindent
Rewriting \eqref{SYS}-\eqref{SYSBC2} in the form
\begin{equation}
\label{SYSM}
\begin{cases}
\displaystyle
|\partial_t u|^\rho\partial_{tt} u+A\partial_{tt}u+A\partial_t u
+ Au + \int_0^\infty \mu(s) A\eta(s)\,\d s + f(u)=h,\\
\partial_t\eta=T \eta+ \partial_t u,
\end{cases}
\end{equation}
the following result is proved in \cite{CMP}.

\begin{theorem}
\label{existence}
Let $h\in\H^{-1}$ and $\rho\in[0,4]$. Then system~\eqref{SYSM}
generates a solution semigroup $S(t):\HH\to\HH$
that satisfies the joint continuity
$$(t,z)\mapsto S(t)z\in\C([0,\infty)\times\HH,\HH).$$
Besides,
given any initial data $z=(u_0,v_0,\eta_0)\in\HH$ and denoting the corresponding solution
by
$$(u(t),\partial_t u(t),\eta^t)=S(t)z,$$
we have the explicit representation formula
\begin{equation}
\label{REP}
\eta^t(s)
=\begin{cases}
u(t)-u(t-s) & 0<s\leq t,\\
\noalign{\vskip1mm}
\eta_0(s-t)+u(t)-u_0 & s>t.
\end{cases}
\end{equation}
\end{theorem}

Moreover, defining the energy at time $t$ of the solution $S(t)z$ as
\begin{equation}
\label{EE}
E(t)=\frac12\|S(t)z\|_\HH^2=\frac12\big[\|u(t)\|_1^2+\|\partial_t u(t)\|_{1}^2+\|\eta^t\|_{\M}^2 \big],
\end{equation}
we have (see \cite{CMP})

\begin{proposition}
\label{propppppy}
The uniform estimates
$$
E(t)+\|\partial_{tt} u(t)\|_1\leq \Q(R)
$$
and
$$
\int_0^\infty\|\partial_t u(t)\|_1^2\,\d t\leq \Q(R)
$$
hold for every initial data $z\in\HH$ with $\|z\|_\HH\leq R$.
\end{proposition}

Finally, we show the existence of a gradient system structure.
We first recall the definition.

\begin{definition}
\label{DefA}
A function $\LL\in \C(\HH,\R)$ is called a {\it Lyapunov functional} if
\begin{itemize}
\item[(i)] $\LL(\zeta)\to\infty$ if and only if $\|\zeta\|_\HH\to\infty$;
\item[(ii)] $\LL(S(t)z)$ is nonincreasing for any $z\in \HH$;
\item[(iii)] if $\LL(S(t)z)=\LL(z)$ for all $t>0$, then $z$ is a stationary point
for $S(t)$.
\end{itemize}
If there exists a Lyapunov functional, then $S(t)$ is called a {\it gradient system}.
\end{definition}

\begin{proposition}
\label{GS}
$S(t)$ is a gradient system on $\HH$.
\end{proposition}

\begin{proof}
For $\zeta=(u,v,\eta)$, let us define
$$
\LL(\zeta)=\frac1{\rho+2}\int_\Omega|v|^{\rho+2}\,\d \x +\frac12\|\zeta\|_\HH^2 +\l F(u),1\r-\l h,u\r.
$$
In light of \eqref{GROW} and \eqref{extra2}, it is readily seen that
\begin{equation}
\label{est_lyapunov}
\frac{\nu}{4} \|\zeta\|_\HH^2 -c_{f,h}\leq\LL(\zeta)\leq c \|\zeta\|_\HH^2\big(1+\|\zeta\|_\HH^4\big)+\|h\|^2_{-1},
\end{equation}
where
$$
c_{f,h}=M_f+\frac{1}{\nu}\|h\|_{-1}^2.
$$
This proves (i). For sufficiently regular initial data $z$,
testing system~\eqref{SYSM} with $(\partial_t u,\eta)$ in $\H\times\M$ and recalling \eqref{TTT},
we get
\begin{equation}
\label{disL}
\ddt \LL(S(t)z)+\|\partial_t u(t)\|_1^2=\l T\eta^t,\eta^t\r_\M\leq 0,
\end{equation}
so establishing (ii). To prove (iii), we note that if $\LL(S(t)z)$ is constant, it follows that
$$\|\partial_t u(t)\|_1^2=\l T\eta^t,\eta^t\r_\M=0.$$
Therefore $\partial_t u(t)\equiv 0$, so that $u(t)=u_0$ for all $t$.
In particular, the second equation of~\eqref{SYSM} reduces to
$$\partial_t\eta=T \eta,$$
and a multiplication by $\eta$ gives 
$$\|\eta^t\|_{\M}=\|\eta_0\|_{\M},\quad\forall t\geq 0.$$
On the other hand, we learn from~\eqref{REP} that
$$
\eta^t(s)
=\begin{cases}
0 & 0<s\leq t,\\
\noalign{\vskip1mm}
\eta_0(s-t) & s>t,
\end{cases}
$$
thus, in light of \eqref{nece}
$$
\|\eta_0\|_{\M}^2=\|\eta^t\|_{\M}^2=\int_0^\infty\mu(t+s) \|\eta_0(s)\|^2_1\,\d s
\leq C\e^{-\delta t}\|\eta_0\|_{\M}^2,
$$
which forces the equality $\eta_0=0$.
In conclusion,
$$S(t)z=z=(u_0,0,0),$$
meaning that $z$ is a stationary point.
\end{proof}

\section{Dissipativity}
\label{SecABS}

\noindent
The dissipativity of $S(t)$ follows from the existence of a bounded absorbing set.
This is a straightforward consequence of the next result.

\begin{theorem}
\label{UNI-ABS}
There exists $\omega>0$ such that
$$E(t)\leq \Q(R)\e^{-\omega t}+R_0$$
whenever $E(0)\leq R$,
having set
$$R_0=\frac{4}{\nu}\big(c_{f,h}+M_f\big)
=\frac{4}{\nu}\Big(2M_f+\frac{1}{\nu}\|h\|_{-1}^2\Big).$$
\end{theorem}

\begin{remark}
In light of the theorem, every ball $\BB$ of $\HH$ centered at zero with radius strictly greater than $\sqrt{2R_0}$
is a (bounded) absorbing set for $S(t)$. Recall that $\BB$ is called an absorbing set if for every bounded set
$\B\subset\HH$ there exists a time $t_\B\geq 0$ such that
$$S(t)\B\subset\BB,\quad\forall t\geq t_\B.
$$
\end{remark}

If $R_0=0$ the exponential decay of the energy occurs.

\begin{corollary}
\label{corasdf}
Let $h=0$ and $f$ satisfy \eqref{DISS1}-\eqref{DISS2}
with $m_f=0$. Then
$$E(t)\leq \Q(R)\e^{-\omega t}$$
whenever $E(0)\leq R$.
\end{corollary}

As a first step, we prove the result in a weaker form, allowing the
(exponential) decay rate to depend on $R$.

\begin{lemma}
\label{lemmaUNI-ABS}
For every $R\geq 0$ there exists a constant $\delta=\delta_R>0$ such that
$$E(t)\leq \Big[\Q(R)E(0)+\frac4\nu\|h\|_{-1}^2\Big]\e^{-\delta t}+R_0$$
whenever $E(0)\leq R$.
\end{lemma}

\begin{proof}
For a fixed $\|z\|_\HH\leq R$, let
$$\LL(t)=\LL(S(t)z)$$
be the Lyapunov functional of Proposition~\ref{GS}.
Then, we introduce the further functionals
\begin{align*}
\Psi(t)&=\int_0^\infty\Big(\int_s^\infty\mu(y)\,\d y\Big)\|\eta^t(s)\|_1^2\,\d s\\
\Phi(t)&=\frac12\|u(t)\|_1^2+\l \partial_t u(t),u(t)\r_1+\frac1{\rho+1}\l|\partial_tu(t)|^{\rho}\partial_tu(t),u(t)\r.
\end{align*}
Arguing as in the proof of Lemma~\ref{Abound_SYS} in Appendix,
\begin{equation}
\label{gazzola1}
\ddt\Psi +\frac12\|\eta\|^2_{\M}
\leq 2\Theta^2\kappa\|\partial_tu\|^2_{1}.
\end{equation}
A multiplication of the first equation of \eqref{SYSM} by $u$ gives
\begin{align*}
&\ddt\Phi
+\|u\|^2_1 +\l f(u),u\r-\l h,u\r\\
&=-\int_0^\infty\mu(s)\l \eta(s),u\r_1\,\d s+\|\partial_{t} u\|_1^2
+\frac{1}{\rho+1}\int_\Omega|\partial_t u|^{\rho+2}\,\d \x,
\end{align*}
and from \eqref{extra1} we obtain
\begin{align}
\label{gazzola2}
&\ddt\Phi+
\frac{1+\nu}{2}\|u\|^2_1 +\l F(u),1\r-\l h,u\r\\
&\leq
-\int_0^\infty\mu(s)\l \eta(s),u\r_1\,\d s+\|\partial_{t} u\|_1^2
+\frac{1}{\rho+1}\int_\Omega|\partial_t u|^{\rho+2}\,\d \x+M_f.
\nonumber
\end{align}
Fixing $\eps>0$ such that
$$1-2\eps\Theta^2\kappa\geq \frac12,$$
and for $\delta>0$ to be properly chosen later,
we consider the functional
$$\E(t)=\LL(t)+\eps\Psi(t)+\delta\Phi(t).$$
Collecting \eqref{disL} and \eqref{gazzola1}-\eqref{gazzola2}, we end up with
\begin{align*}
&\ddt \E+
\delta \LL+\delta^2\Phi+\frac{\delta}{2}(\nu-\delta)\|u\|_1^2
+\frac12(1-3\delta)\|\partial_t u\|_1^2+\frac12(\eps-\delta)\|\eta\|^2_{\M}\\
&\leq \delta M_f-\delta\int_0^\infty\mu(s)\l \eta(s),u\r_1\,\d s
+\delta^2\l \partial_{t} u,u\r_1\\
&\quad+\delta\Big(\frac{1}{\rho+1}+\frac{1}{\rho+2}\Big)\int_\Omega|\partial_t u|^{\rho+2}\,\d \x
+\frac{\delta^2}{\rho+1}\l |\partial_t u|^\rho\partial_t u, u\r.
\end{align*}
We estimate the right-hand side above in the following way.
For $\delta>0$ sufficiently small, standard computations entail
$$-\delta\int_0^\infty\mu(s)\l \eta(s),u\r_1\,\d s\leq \delta\sqrt{\kappa}\|\eta\|_{\M}\|u\|_1\leq
\frac{\delta}{8}(\nu-\delta)\|u\|_1^2+c\delta\|\eta\|^2_{\M}$$
and
$$\delta^2\l \partial_{t} u,u\r_1\leq
\frac{\delta}{8}(\nu-\delta)\|u\|_1^2+\frac14\|\partial_t u\|_1^2.$$
Moreover, by Proposition~\ref{propppppy} and the embedding $H^1\subset L^{\frac65(\rho+1)}(\Omega)$,
$$\delta\Big(\frac{1}{\rho+1}+\frac{1}{\rho+2}\Big)\int_\Omega|\partial_t u|^{\rho+2}\,\d \x
+\frac{\delta^2}{\rho+1}\l |\partial_t u|^\rho\partial_t u, u\r\leq \delta\Q(R)\|\partial_{t}u\|^2_1
+\frac\delta4(\nu-\delta)\|u\|^2_1.$$
Therefore, we arrive at
$$
\ddt \E+
\delta \LL+\delta^2\Phi
+\Big(\frac14-\frac{3\delta}{2}-\delta\Q(R)\Big)\|\partial_t u\|_1^2
+\Big(\frac\eps2-c\delta\Big)\|\eta\|^2_{\M}
\leq \delta M_f.
$$
Hence, we can choose $\delta=\delta_R$ small enough that
$$
\ddt \E+
\delta\E+\eps\Big(\frac14\|\eta\|^2_{\M}-\delta\Psi\Big)
\leq \delta M_f.
$$
Actually, since by \eqref{NEC1}
\begin{equation}
\label{psi}
0\leq\Psi\leq \Theta\|\eta\|_{\M}^2\leq 2\Theta E,
\end{equation}
up to further reducing $\delta$, we get
$$
\ddt \E+\delta\E\leq \delta M_f,
$$
and an application of the Gronwall lemma leads to
\begin{equation}
\label{la1}
\E(t)\leq \E(0)\e^{-\delta t} + M_f.
\end{equation}
We now prove that
\begin{equation}
\label{la2}
\frac{\nu}{4} E-c_{f,h}\leq \E\leq \Q(R)E+\|h\|_{-1}^2.
\end{equation}
Indeed,
\begin{align*}
|\Phi|&\leq\|u\|_1^2+\frac12\|\partial_t u\|^2_1
+\frac1{\rho+1}\int_\Omega|\partial_tu|^{\rho+1}|u|\,\d \x\\
&\leq
\|u\|_1^2+\frac12\|\partial_t u\|^2_1
+\frac1{\rho+1}\|\partial_t u\|_{L^{\frac65(\rho+1)}}^{\rho+1}\|u\|_{L^6}\\
&\leq
\|u\|_1^2+\frac12\|\partial_t u\|^2_1
+c\|\partial_tu\|_1^\rho\big(\|\partial_tu\|_1^2+\|u\|_1^2\big)\leq \Q(R)E.
\end{align*}
Therefore, on account of \eqref{est_lyapunov}
and \eqref{psi},
we readily get
$$\E(t)\leq \Q(R)E(t)+\|h\|_{-1}^2.$$
Besides,
$$\E\geq \LL-\delta|\Phi|\geq \frac{\nu}{2} E-c_{f,h}-\delta\Q(R)E,$$
hence, possibly by further reducing $\delta$ in dependence of $R$,
we obtain
$$\E\geq \frac{\nu}{4} E-c_{f,h}.$$
The claim follows from \eqref{la1} and \eqref{la2}.
\end{proof}

\begin{proof}[Proof of Theorem \ref{UNI-ABS}]
Let  $\|z\|_\HH\leq R$ for some $R\geq 0$. Then, we infer from Lemma~\ref{lemmaUNI-ABS}
the existence of $t_R\geq 0$ such that
$$E(t_R)\leq 1+R_0,$$
and a further application of Lemma~\ref{lemmaUNI-ABS}
yields
$$
E(t)=\frac12\|S(t-t_R)S(t_R)z\|_\HH^2
\leq \Q(R_0)\e^{\omega t_R}\,\e^{-\omega t}+R_0,\quad\forall t>t_R,
$$
where $\omega=\delta_{1+R_0}$.
At the same time, again by Lemma~\ref{lemmaUNI-ABS},
$$E(t)\leq \Q(R)+R_0,\quad\forall t\leq t_R.$$
Collecting the two inequalities we are done.
\end{proof}

\section{Main Results}
\label{SecMS}

\begin{theorem}
\label{attractor}
The semigroup $S(t)$ possesses the global attractor $\A$.
\end{theorem}

By definition, the global attractor of $S(t)$  is
the unique compact set $\A\subset\HH$ which is at the same time
fully invariant and attracting for the semigroup
(see e.g.\ \cite{BV,HAL,HAR,TEM}). Namely,
\begin{itemize}
\item[(i)] $S(t)\A=\A$ for every $t\geq 0$; and
\smallskip
\item[(ii)] for every bounded set $\B\subset\HH$
$$\lim_{t\to\infty}\,\dist_{\HH}(S(t)\B,\A)=0,$$
\end{itemize}
where $\dist_{\HH}$ denotes the standard Hausdorff semidistance in $\HH$.
We also recall that, for an arbitrarily fixed $\tau\in\R$, the global attractor can be
given the form (see \cite{HAR})
$$\A=\big\{\zeta(\tau): \zeta \, \text{\cbt}\big\},
$$
where a complete bounded trajectory ({\cbt}) of the semigroup
is a function $\zeta\in\C_{\rm b}(\R,\HH)$
satisfying
$$\zeta(\tau)=S(t)\zeta(\tau-t),\quad\forall t\geq 0,\,\forall \tau\in\R.$$

According to \cite{CH,HAL},
the existence of a Lyapunov function (Proposition~\ref{GS})
ensures that $\A$ coincides with the unstable manifold of the set ${\mathbb S}$
of equilibria of $S(t)$, which is compact, nonempty and
made of all vectors $z^\star=(u^\star,0,0)$ with $u^\star$ solution to the elliptic equation
$A u^\star +f(u^\star)=h$. That is,
$$\A=\big\{\zeta(0):\, \zeta \text{ is a {\cbt} and } \lim_{\tau\to-\infty}\|\zeta(\tau)-{\mathbb S}\|_{\HH}=0\big\}.
$$
Moreover, the following result holds.

\begin{proposition}
\label{STAT}
Any {\cbt} $\zeta=(u,\pt u,\eta)$ fulfills the relation
$$\lim_{\tau\to\pm\infty}\|\zeta(\tau)-{\mathbb S}\|_{\HH}=0.$$
In particular,
$$
\lim_{\tau\to \pm\infty} \big[\|\pt u(\tau)\|_1+\|\eta^\tau\|_{\M}\big]=0.
$$
\end{proposition}

\begin{corollary}
If ${\mathbb S}$
is discrete, there exist $z^\star,w^\star\in {\mathbb S}$ such that
$\zeta(\tau)\to z^\star$ in $\HH$ as $\tau\to\infty$ and $\zeta(\tau)\to w^\star$ in $\HH$ as $\tau\to-\infty$.
\end{corollary}

On the other hand, ${\mathbb S}$ might as well be a continuum
(e.g.\ if $F$ is a double-well potential, see \cite{HAR}). In such a case,
the convergence of a given trajectory to a single equilibrium cannot be
predicted, and is false in general. Nonetheless, if $f$ is real analytic,
there is a well-known tool which can be used in order to guarantee the
convergence of trajectories to equilibria: the
{\L}ojasiewicz-Simon inequality (see e.g.\ \cite{HJ}).

\smallskip
Coming to the regularity of the attractor, we have

\begin{theorem}
\label{attrH1}
The {\it global attractor} $\A$ of $S(t)$ is bounded in $\HH^1$.
\end{theorem}

Theorem \ref{attractor} and Theorem \ref{attrH1} subsume the main Theorem \ref{MAIN} stated in the introduction.

Observe that, as ${\mathbb S}\subset\A$, if $h\in\H$ without any further assumption
we cannot have more than $\H^2$-regularity for the first component. Thus the
inclusion $\A\subset \HH^1$ is optimal.

\begin{proposition}
\label{oppy}
For every {\cbt} $\zeta=(u,\pt u,\eta)$
the formal equality \eqref{ForE}
holds true for every $t\in\R$.
In particular, it follows that $\eta^t\in\D(T)$ for all $t$.
\end{proposition}

A direct consequence of the proposition
is the next corollary, whose proof is left to the reader.

\begin{corollary}
\label{corry}
Given $u\in\C_{\rm b}(\R,\H^1)\cap\C_{\rm b}^1(\R,\H^1)$
and defining $\eta=\eta^t(s)$ for all real $t$ by the formula \eqref{ForE},
the vector
$\zeta=(u,\pt u,\eta)$ is a {\cbt}
if and only if
$u$ solves  the equation
$$|\partial_t u(t)|^\rho\partial_{tt} u(t)+A\partial_{tt}u(t)+A\partial_t u(t)
+(1+\kappa)A u(t) -\int_0^\infty \mu(s) A u(t-s)\,\d s + f(u(t))=h$$
for every $t\in\R$.
\end{corollary}

The proofs of the results stated above will be carried out in the subsequent sections.

\section{Existence of the Global Attractor}
\label{SecEGA}

\noindent
In what follows, let $\rho\in [0,4)$.
Besides, let $\BB$ be a given bounded absorbing set, whose existence
is guaranteed by Theorem~\ref{UNI-ABS}.
The main result of the section is

\begin{proposition}
\label{Attr3}
For any $t\geq0$, there exists a compact set $\K(t)\subset\HH$ such that
$$
\dist_{\HH}(S(t)\BB,\K(t))\leq c\e^{-\omega t}
$$
for some $c\geq 0$ and $\omega>0$ depending only on $\BB$.
\end{proposition}

Proposition~\ref{Attr3} tells that $S(t)$ is {\it asymptotically compact}.
Hence, invoking a general result of the theory of dynamical systems
(see e.g.\ \cite{BV,CV,HAL,HAR,TEM}), $S(t)$ possesses
the global attractor $\A$. This establishes the proof of Theorem~\ref{attractor}.

In order to prove Proposition~\ref{Attr3}, we need a suitable decomposition of $f$.

\begin{lemma}
\label{DECOfff}
The nonlinearity $f$ admits the decomposition
$$
f(s)=f_0(s)+f_1(s)
$$
for some $f_0,f_1$ with the following properties:
\begin{itemize}
\item[$\bullet$] $f_1$ is Lipschitz continuous with $f_1(0)=0$;
\smallskip
\item[$\bullet$] $f_0$ vanishes inside $[-1,1]$ and
fulfills the critical growth restriction
$$|f_0(u)-f_0(v)|\leq c|u-v|(|u|+|v|)^4$$
\item[$\bullet$] $f_0$
fulfills for every $s\in\R$ the bounds
$$
f_0(s)s\geq F_0(s)\geq 0,
$$
where $F_0(s)=\int_0^s f_0(y)\,\d y$.
\end{itemize}
\end{lemma}

\begin{proof}
Set $\alpha=\lambda_1(1-\nu)$ and fix $\beta\in(\alpha,\lambda_1)$.
Collecting \eqref{DISS1}-\eqref{DISS2}, we know that
\begin{equation}
\label{kkaa}
f(s)s\geq -\alpha s^2-2m_f,\quad\forall s\in\R.
\end{equation}
Let $k\geq 1$ large enough to have
\begin{equation}
\label{ka}
(\beta-\alpha)s^2-2m_f\geq 0,\quad\forall \, |s|\geq k.
\end{equation}
Choosing then any smooth function
$\varrho:\R\to [0,1]$ satisfying $\varrho'(s)s\geq 0$ and
$$\varrho(s)=
\begin{cases}
0 & \text{if } |s|\leq k,\\
1 & \text{if } |s|\geq k+1,
\end{cases}
$$
define
$$
f_0(s)=\varrho(s)[f(s)+\beta s]
\and
f_1(s)=[1-\varrho(s)]f(s)-\beta\varrho(s)s.
$$
In light of \eqref{kkaa}-\eqref{ka},
it is immediate to check that
$f_0(s)s\geq 0$, implying in turn $F_0(s)\geq 0$.
We are left to prove the estimate $f_0(s)s\geq F_0(s)$.
We limit ourselves to discuss the case $s>0$, being the other one analogous.
If $s<k$, then $f_0(s)=F_0(s)=0$ by the very definition of $\varrho$.
If $s\geq k$, using again \eqref{kkaa}-\eqref{ka}
we infer that
$$f(y)+\beta y\geq 0,\quad\forall y\in[k,s].$$
Hence
\begin{align*}
F_0(s)&=\int_k^s\varrho(y)[f(y)+\beta y]\,\d y\\
&\leq
\varrho(s)\int_k^s[f(y)+\beta y]\,\d y\\
&=\varrho(s)\Big[F(s)+\frac{\beta}{2}s^2\Big]-\varrho(s)\Big[F(k)+\frac{\beta}{2}k^2\Big].
\end{align*}
Exploiting \eqref{DISS1}-\eqref{DISS2} and \eqref{ka}, we get
\begin{align*}
F_0(s)
&\leq\varrho(s)\Big[f(s)s+\frac{\alpha}{2}s^2+m_f+\frac{\beta}{2}s^2\Big]-\frac{\varrho(s)}{2}[(\beta-\alpha)k^2-2m_f]\\
&=f_0(s)s-\frac{\varrho(s)}{2}[(\beta-\alpha)s^2-2m_f]
-\frac{\varrho(s)}{2}[(\beta-\alpha)k^2-2m_f]
\leq f_0(s)s.
\end{align*}
This concludes the proof.
\end{proof}

Defining now
$$\sigma=\min\Big\{\frac13,\frac{4-\rho}{2}\Big\},$$
the following result holds.

\begin{lemma}
\label{Attr2}
For any $t\geq0$, there exists a closed bounded set ${\mathcal B}_\sigma(t)\subset\HH^{\sigma}$ such that
$$
\dist_{\HH}(S(t)\BB,{\mathcal B}_\sigma(t))\leq c\e^{-\omega t},
$$
for some constants $c\geq 0$ and $\omega>0$
depending only on $\BB$.
\end{lemma}

\begin{proof}
We write $f=f_0+f_1$ as in Lemma~\ref{DECOfff}.
For an arbitrarily fixed $z\in\BB$, let
$$(\hat v(t),\partial_t\hat v(t),\hat\xi^t)
\and
(\hat w(t),\partial_t\hat w(t),\hat\psi^t)$$
be the solutions at time $t>0$ to
the problems
\begin{equation}
\label{SYYSuno}
\begin{cases}
\displaystyle
|\partial_t\hat v|^\rho\partial_{tt}\hat v+A \partial_{tt}\hat v
+A\partial_t\hat v
+A\hat v+\int_0^\infty\mu(s)A\hat\xi(s)\,\d s+f_0(\hat v)=0,\\
\partial_t\hat\xi=T\hat\xi+\partial_t\hat v,\\
\noalign{\vskip1.5mm}
(\hat v(0),\partial_t\hat v(0),\hat \xi^0)=z,
\end{cases}
\end{equation}
and
\begin{equation}
\label{SYYSdue}
\begin{cases}
\displaystyle
|\partial_t u|^\rho \partial_{tt} u-|\partial_t\hat v|^\rho \partial_{tt}\hat v+A \partial_{tt}\hat w
+A\partial_t\hat w+A\hat w+\int_0^\infty\mu(s)A\hat \psi(s)\,\d s=g\\
\partial_t\hat\psi=T\hat\psi+\partial_t\hat w,\\
\noalign{\vskip1.5mm}
(\hat w(0),\partial_t\hat w(0),\hat\psi^0)=0,
\end{cases}
\end{equation}
having set
$$g=h-f_0(u)+f_0(\hat v)-f_1(u).$$
In what follows, the generic constant $c\geq 0$
is independent of the
choice of $z\in\BB$.

\smallskip
Concerning system \eqref{SYYSuno},
since the forcing term is null and
$f_0(v)v\geq F_0(v)\geq 0$, an application of
Lemma~\ref{lemmaUNI-ABS} yields the
exponential decay
\begin{equation}
\label{estimate_hatv}
\big\|\big(\hat v(t),\partial_t\hat v(t),\hat\xi^t\big)\big\|_\HH\leq c\|z\|_{\HH}\e^{-\omega t},
\end{equation}
for some $c\geq 0$ and $\omega>0$, depending only on $\BB$.
Furthermore, a multiplication of the first equation of \eqref{SYYSuno} by $\partial_{tt} \hat v$  gives
\begin{align*}
\|\partial_{tt} \hat v\|^2_1&\leq\l|\partial_t \hat v|^\rho \partial_{tt} \hat v,\partial_{tt} \hat v\r+\|\partial_{tt} \hat v\|^2_1\\
&=-\l \partial_t \hat v,\partial_{tt} \hat v\r_1-\l \hat v,\partial_{tt} \hat v\r_1
-\int_0^\infty\mu(s)\l\hat \xi(s),\partial_{tt} \hat v\r_1\,\d s-\l f_0(\hat v),\partial_{tt} \hat v\r.
\end{align*}
By the growth assumption on $f_0$,
$$-\l f_0(\hat v),\partial_{tt} \hat v\r
\leq \|f_0(\hat v)\|_{L^{6/5}}\|\partial_{tt}\hat v\|_{L^6}\leq c\big(1+\|\hat v\|^5_1\big)\|\partial_{tt}\hat v\|_1.$$
Moreover,
$$-\int_0^\infty\mu(s)\l\hat \xi(s),\partial_{tt} \hat v\r_1\,\d s\leq
\|\partial_{tt}\hat v\|_1\int_0^\infty\mu(s)\|\hat \xi(s)\|_1\,\d s,
$$
and
$$\int_0^\infty\mu(s)\|\hat \xi(s)\|_1\,\d s
\leq\sqrt{\kappa}\,\|\hat \xi\|_\M.
$$
Thus, we infer from \eqref{estimate_hatv} that
$$
\|\partial_{tt} \hat v\|^2_1
\leq \big(\|\partial_t \hat v\|_{1}+\|\hat v\|_1
+\sqrt{\kappa}\|\hat \xi\|_\M+c+c\|\hat v\|^5_1\big)\|\partial_{tt} \hat v\|_1\leq \frac12\|\partial_{tt} \hat v\|^2_1+c
$$
implying the bound
\begin{equation}
\label{est2_hatv}
\|\partial_{tt} \hat v\|_1\leq c.
\end{equation}

\smallskip
Concerning system \eqref{SYYSdue},
introducing the energy
$$\hat E_\sigma(t)=\frac12\big\|\big(\hat w(t),\partial_t\hat w(t),\hat\psi^t\big)\big\|_{\HH^\sigma}^2,$$
we want to prove the estimate
\begin{equation}
\label{estimate_hatw}
\hat E_\sigma(t)\leq\e^{ct}.
\end{equation}
To this aim, we multiply
the first equation of \eqref{SYYSdue}
by $A^\sigma\partial_t\hat w$,
and the second one by $\hat\psi$ in $\M^\sigma$, so to get
$$\ddt \hat E_\sigma+\|\partial_t\hat w\|^2_{1+\sigma}\leq
\l-|\partial_t u|^\rho\partial_{tt}u+|\partial_t\hat v|^\rho\partial_{tt}\hat v,A^\sigma\partial_t\hat w\r
+\l g,A^\sigma\partial_t\hat w\r.
$$
Observe that
\begin{align*}
|g|=|h-f_0(u)+f_0(\hat v)-f_1(u)|
\leq |h|+c|\hat w|(|u|+|\hat v|)^4+c(1+|u|).
\end{align*}
Thus, the Sobolev embeddings
$$\H^{1+\sigma}\subset L^{\frac{6}{1-2\sigma}}(\Omega)\and \H^{1-\sigma}\subset
L^{\frac{6}{1+2\sigma}}(\Omega)$$
yield
\begin{align}
\label{hatg}
\l g,A^{\sigma}\partial_t\hat w\r\leq&
\|h\|\|A^\sigma\partial_t\hat w\|
+c(\|u\|_{L^6}+\|\hat v\|_{L^6})^4\|\hat w\|_{L^{6/(1-2\sigma)}}
\|A^{\sigma}\partial_t\hat w\|_{L^{6/(1+2\sigma)}}\\
\nonumber
&+c(1+\|u\|)\|A^\sigma\partial_t\hat w\|\\
\nonumber
\leq &c\|\partial_t\hat w\|_{1+\sigma}+c(\|u\|_1+\|\hat v\|_1)^4\|\hat w\|_{1+\sigma}
\|\partial_t\hat w\|_{1+\sigma}\\
\nonumber
&+c(1+\|u\|_1)\|\partial_t\hat w\|_{1+\sigma}\\
\leq & c \hat E_\sigma+c.
\nonumber
\end{align}
Besides,
since $\frac{3\rho}{2-\sigma}\leq 6$, from the embedding $\H^1\subset L^{\frac{3\rho}{2-\sigma}}(\Omega)$ we find
$$
\l-|\partial_t u|^\rho\partial_{tt}u,A^{\sigma}\partial_t\hat w\r
\leq
\|\partial_t u\|^\rho_{L^{3\rho/(2-\sigma)}}\|\partial_{tt}u\|_{L^6}
\|A^{\sigma}\partial_t\hat w\|_{L^{6/(1+2\sigma)}}
\leq\|\partial_t u\|_{1}^\rho\|\partial_{tt}u\|_{1}
\|\partial_t\hat w\|_{1+\sigma}
$$
and, analogously,
$$\l|\partial_t\hat v|^\rho\partial_{tt}\hat v,A^{\sigma}\partial_t\hat w\r\leq
\|\partial_t \hat v\|_{1}^\rho\|\partial_{tt}\hat v\|_{1}
\|\partial_t\hat w\|_{1+\sigma}.$$
Therefore, in light of  Proposition~\ref{propppppy}, \eqref{estimate_hatv} and \eqref{est2_hatv}, we have
$$
\l-|\partial_t u|^\rho\partial_{tt}u+|\partial_t\hat v|^\rho\partial_{tt}\hat v,A^{\sigma}\partial_t\hat w\r
\leq \|\partial_t\hat w\|_{1+\sigma}^2+c.
$$
Collecting the above inequalities we arrive at
$$
\ddt \hat E_\sigma\leq c \hat E_\sigma+c.
$$
Recalling that $\hat E_\sigma(0)=0$, by
the Gronwall lemma we obtain the sought inequality \eqref{estimate_hatw}.
This finishes the proof.
\end{proof}

Lemma~\ref{Attr2} is not quite enough to conclude, since the embedding $\HH^\sigma\subset\HH$
is not compact (see \cite{PZ}). Hence, a further argument is needed.

\begin{proof}[Proof of Proposition \ref{Attr3}]
In the previous notation, for any $z\in\BB$ and
any fixed $t\geq 0$ we set
$$\Xi_t=\bigcup_{z\in\BB} \hat\psi^t.$$
Exploiting the representation formula for $\hat\psi^t$
$$
\hat\psi^t(s)
=\begin{cases}
\hat w(t)-\hat w(t-s) & 0<s\leq t,\\
\noalign{\vskip1mm}
\hat w(t) & s>t,
\end{cases}
$$
and taking into account that $\partial_t \hat  w\in L^\infty(0,\infty;\H^1)$,
we learn that $\Xi_t\subset\D(T)$, and by elementary computations we obtain
$$\sup_{z\in\BB}\|T\hat\psi^t\|_{\M}\leq c
\qquad\text{and}\qquad\sup_{z\in\BB}\|\hat\psi^t(s)\|_1^2\leq c.$$
Besides, by \eqref{estimate_hatw},
$$\sup_{z\in\BB}\| \hat\psi^t\|_{\M^{\sigma}}\leq\Q(t).$$
Since
$$s\mapsto c\mu(s)\in L^1(\R^+),$$
we infer from Lemma~\ref{lemmaPZ} that $\Xi_t$ is precompact in $\M$.
At this point, exploiting \eqref{estimate_hatw} again,
let $\B(t)$ be the closed ball of $\H^{1+\sigma}\times \H^{1+\sigma}$
centered at zero of a suitable radius $\Q(t)$ such that
$$\sup_{z\in\BB}\|\big(\hat w(t),\partial_t\hat w(t)\big)\big\|_{\H^{1+\sigma}\times \H^{1+\sigma}}\leq\Q(t).
$$
Finally, define
$$\K(t)=\B(t)\times\overline\Xi_t,$$
the bar standing for the closure in $\M$. Then
$\K(t)$ is compact in $\HH$ and fulfills the claim.
\end{proof}

\begin{remark}
Actually, relying on the gradient system structure of $S(t)$ provided by Proposition~\ref{GS},
one could prove the existence of the global attractor without passing through the
existence of a bounded absorbing set, which is then recovered as a byproduct (see e.g.\ \cite{visco, HAL}).
The disadvantage of this scheme is that it does not provide any estimate of the entering time
into the absorbing set.
\end{remark}

\section{Further Regularity}
\label{SecFR}

\begin{proposition}
\label{AttrH13}
The {\it global attractor} $\A$ is bounded in $\HH^{\sigma}$.
\end{proposition}

\begin{proof}
The global attractor $\A$, being fully invariant, is contained in every closed attracting set.
Hence, to prove the lemma it is enough to exhibit a (closed) ball
$\BB_\sigma\subset \HH^\sigma$ which
attracts the bounded absorbing set $\BB$.
Indeed, by applying
Lemma~\ref{teo_win} with $r=\sigma$, we will show that
$$
\dist_{\HH}(S(t)\BB,\BB_\sigma)\leq c\e^{-\varkappa t},
$$
for some $\varkappa>0$.
To this end, let
$z\in\BB$ be fixed, and let $y\in\BB$ and $x\in\HH^{\sigma}$ be any pair satisfying $y+x=z$.
We define the operators
$$V_{z}(t)y=(v(t),\partial_t v(t),\xi^t)\qquad\text{ and }\qquad U_{z}(t)x=(w(t),\partial_tw(t),\psi^t),$$
where $(v(t),\partial_t v(t),\xi^t)$ and $(w(t),\partial_t w(t),\psi^t)$ solve systems \eqref{SYYSuno} and
\eqref{SYYSdue} without the {\it hats}, with initial data
$$(v(0),\partial_t v(0),\xi^0)=y\and (w(0),\partial_t w(0),\psi^0)=x.$$

\smallskip
Condition (i) of Lemma~\ref{teo_win}
holds by construction, while (ii) follows by the exponential decay \eqref{estimate_hatv},
which now reads
\begin{equation}
\label{estimatevvvv}
\|(v(t),\partial_t v(t),\xi^t)\|_{\HH}=\|V_z(t)y\|_{\HH}\leq c\|y\|_\HH\e^{-\omega t}.
\end{equation}
Arguing as in the proof of \eqref{est2_hatv} we also get
\begin{equation}
\label{est2v}
\|\partial_{tt} v\|_1\leq c.
\end{equation}
In order to prove (iii), we set
$$E_\sigma(t)=\frac12\|U_{z}(t)x\|_{\HH^\sigma}^2.$$
An application of Lemma \ref{Abound_SYS}
provides a functional
$\Lambda_{\sigma}$
satisfying
\begin{equation}
\label{eqis}
\frac12E_\sigma\leq \Lambda_\sigma\leq 2E_\sigma,
\end{equation}
jointly with the differential inequality
$$
\ddt\Lambda_\sigma+\delta E_\sigma
\leq\l \gamma,\partial_tw\r_\sigma+\delta\l \gamma,w\r_\sigma,
$$
which holds for all $\delta>0$ small enough.
Here, $\gamma$ is defined by
$$\gamma=g-|\partial_t u|^\rho\partial_{tt}u+|\partial_t v|^\rho\partial_{tt}v$$
where
$$g=h-f_0(u)+f_0(v)-f_1(u).$$
We estimate the nonlinearity $g$ as follows: we write
\begin{align*}
|g|&\leq |h|+c|w|(|u|+|v|)^4+c(1+|u|)\\
&\leq|h|+c|w|(|\hat v|+|v|)^4+c(|u|+|v|)|\hat w|^4+c(1+|u|)
\end{align*}
and, with analogous computations as in \eqref{hatg}, we obtain
\begin{align*}
\l g,A^{\sigma} \partial_t w\r\leq&
\|h\|\|A^{\sigma} \partial_t w\|+c(\|v\|_{L^6}+\|\hat v\|_{L^6})^4\|w\|_{L^{6/(1-2\sigma)}}
\|A^{\sigma}\partial_t w\|_{L^{6/(1+2\sigma)}}\\
&+c(\|u\|_{L^6}+\|v\|_{L^6})\|\hat w\|_{L^{6/(1-2\sigma)}}^4\|A^{\sigma} \partial_t w\|_{L^{6/(1+2\sigma)}}
+c(1+\|u\|)\|A^\sigma \partial_t  w\|\\
\leq & c\|\partial_t w\|_{1+\sigma}+
c(\|v\|_1+\|\hat v\|_1)^4\|w\|_{1+\sigma}\|\partial_t w\|_{1+\sigma}\\
&+c(\|u\|_1+\|v\|_1)\|\hat w\|_{1+\sigma}^4\|\partial_t w\|_{1+\sigma}+c(1+\|u\|_1)\|\partial_t w\|_{1+\sigma}.
\end{align*}
Exploiting the decay estimates \eqref{estimate_hatv} and \eqref{estimatevvvv}
together
with \eqref{estimate_hatw}, we arrive at
\begin{equation*}
\l g(t),A^\sigma\partial_t w(t)\r\leq \frac{\delta}{8}\|\partial_t w(t)\|_{1+\sigma}^2+
c\e^{-4\omega t}\big[\|w(t)\|_{1+\sigma}^2+\|\partial_t w(t)\|_{1+\sigma}^2\big]+\Q(t),
\end{equation*}
for some $\Q(\cdot)$ independent of $x$.
Besides, arguing exactly as in the proof of Lemma~\ref{Attr2},
$$
\l-|\partial_t u|^\rho\partial_{tt}u+|\partial_t v|^\rho\partial_{tt}v,A^{\sigma}\partial_t w\r
\leq\frac{\delta}{8}\|\partial_tw\|_{1+\sigma}^2+c.
$$
where we used Proposition \ref{propppppy}, \eqref{estimatevvvv} and \eqref{est2v}.
By analogous computations, we draw
$$
\l \gamma(t),A^\sigma w(t)\r\leq\Big(\frac{\delta}{4}+c\e^{-4\omega t}\Big)\|w(t)\|_{1+\sigma}^2+\Q(t),
$$
for some $\Q(\cdot)$ independent of $x$.
We finally  end up with the differential inequality
\begin{align*}
\ddt\Lambda_\sigma(t)+\frac{\delta}{2} E_\sigma(t)\leq
c\e^{-4\omega t}E_\sigma(t)+\Q(t).
\end{align*}
In light of \eqref{eqis}, we infer from the Gronwall lemma that
$$\|U_z(t)x\|_{\HH^\sigma}\leq c\e^{-\frac{\delta t}{8}}\|x\|_{\HH^\sigma}+\Q(t).$$
This proves (iii).
\end{proof}

A further regularization for $\partial_{tt} u$ will be needed.

\begin{lemma}
\label{tutt}
For initial data $z\in \A$ we have
$$
\|\partial_{tt} u\|_{1+\sigma}\leq c
$$
for some $c>0$ depending only on $\A$.
\end{lemma}

\begin{proof}
A multiplication of the first equation of \eqref{SYSM} by $A^{\sigma}\partial_{tt} u$  gives
\begin{align}
\label{crr}
\|\partial_{tt} u\|^2_{1+\sigma} & \leq -\l|\partial_t u|^\rho \partial_{tt} u,A^\sigma\partial_{tt} u\r
-\l \partial_t u,\partial_{tt} u\r_{1+\sigma}-\l u,\partial_{tt} u\r_{1+\sigma}\\
\nonumber
&\qquad -\int_0^\infty\mu(s)\l\eta(s),\partial_{tt} u\r_{1+\sigma}\,\d s-\l f(u),A^\sigma\partial_{tt} u\r
+\l h,A^\sigma\partial_{tt} u\r.
\end{align}
In order to estimate the terms in the right--hand side,
we exploit the bound
$$\|(u,\partial_t u,\eta)\|_{\HH^{\sigma}}\leq c.$$
Note first that
\begin{align*}
-\l|\partial_t u|^\rho\partial_{tt}u,A^{\sigma}\partial_{tt} u\r
&\leq
\|\partial_t u\|_{L^{3\rho/(2-\sigma)}}^\rho\|\partial_{tt}u\|_{L^{6}}
\|A^{\sigma}\partial_{tt} u\|_{L^{6/(1+2\sigma)}}\\
&\leq
\|\partial_t u\|_{1}^\rho\|\partial_{tt}u\|_{1}
\|\partial_{tt} u\|_{1+\sigma}\\
&\leq c \|\partial_{tt} u\|_{1+\sigma}.
\end{align*}
Moreover, in light of \eqref{GROW},
\begin{align*}
-\l f(u),A^\sigma\partial_{tt} u\r
&\leq \|f(u)\|_{L^{6/(5-2\sigma)}}\|A^\sigma\partial_{tt} u\|_{L^{6/(1+2\sigma)}}\\
&\leq
c\big(1+\|u\|^5_{1+\sigma}\big)\|\partial_{tt}u\|_{1+\sigma}\\
&\leq c\|\partial_{tt}u\|_{1+\sigma},
\end{align*}
and
$$-\int_0^\infty\mu(s)\l\eta(s),\partial_{tt} u\r_{1+\sigma}\,\d s\leq
\sqrt{\kappa}\,\|\eta\|_{\M^\sigma}\|\partial_{tt}u\|_{1+\sigma}.
$$
As a consequence, \eqref{crr} gives
\begin{align*}
\|\partial_{tt}u\|_{1+\sigma}^2
&\leq \big(\|\partial_t u\|_{1+\sigma}+\|u\|_{1+\sigma}
+\sqrt{\kappa}\|\eta\|_{\M^\sigma}+c+\|h\|_{L^{6/(5-2\sigma)}}\big)\|\partial_{tt} u\|_{1+\sigma}\\
&\leq \frac12\|\partial_{tt} u\|^2_{1+\sigma}+c,
\end{align*}
concluding the proof.
\end{proof}

\section{Optimal Regularity of the Attractor}
\label{SecORA}

\noindent
In this section we prove the optimal regularity of the attractor.
The key ingredient is the following lemma.

\begin{lemma}
\label{keylemma}
Given any $r\in[\sigma,1-\sigma]$ the following holds:
$$\A\subset \HH^r\quad\Rightarrow\quad\A\subset \HH^{r+\sigma}.$$
\end{lemma}

\begin{proof}
Knowing that $\A$ is fully invariant and bounded in $\HH^{r}$,
we split the solution $S(t)z$
with $z\in\A$ into the sum
$$S(t)z=L(t)z+K(t)z,$$
where $L(t)z=(v(t),\partial_t v(t),\xi^t)$ and $K(t)z=(w(t),\partial_t w(t),\psi^t)$ solve the systems
$$
\begin{cases}
\displaystyle
A \partial_{tt}  v
+A\partial_t  v
+A  v+\int_0^\infty\mu(s)A \xi(s)\,\d s=0,\\
\partial_t\xi=T \xi+\partial_t  v,\\
\noalign{\vskip1.5mm}
(v(0),\partial_t  v(0),  \xi^0)=z,
\end{cases}
$$
and
$$
\begin{cases}
\displaystyle
A \partial_{tt}  w
+A\partial_t  w+A  w+\int_0^\infty\mu(s)A  \psi(s)\,\d s=\gamma\\
\partial_t\psi=T \psi+\partial_t  w,\\
\noalign{\vskip1.5mm}
(w(0),\partial_t  w(0), \psi^0)=0,
\end{cases}
$$
where
$$\gamma=h-f(u)-|\partial_t u|^\rho\partial_{tt}u.$$
A direct application of Lemma~\ref{Abound_SYS} together with the Gronwall lemma
to the first system shows that the linear semigroup $L(t)$
decays exponentially in $\HH$, i.e.\
\begin{equation}
\label{bound_attr1}
\|L(t)z\|_{\HH}\leq c\e^{-\omega t},
\end{equation}
for some $c=c(\A)>0$ and $\omega>0$.
Defining
$$E_{r+\sigma}(t)=\frac12\|K(t)z\|_{\HH^{r+\sigma}}^2,$$
Lemma \ref{Abound_SYS} applied to
the second system provides the existence of a functional $\Lambda_{r+\sigma}$ satisfying
\begin{equation}
\label{eqis-r}
\frac12E_{r+\sigma}\leq \Lambda_{r+\sigma}\leq 2E_{r+\sigma}
\end{equation}
and
$$
\ddt\Lambda_{r+\sigma}+\delta E_{r+\sigma}
\leq\l \gamma,\partial_t w\r_{r+\sigma}+\delta\l \gamma,w\r_{r+\sigma}
$$
for all $\delta>0$ sufficiently small.
We now exploit the embeddings\footnote{If $r\geq \frac12$ we exploit $\H^{1+r}\subset L^\infty(\Omega)$.}
$$\H^{1+r}\subset L^{\frac{6}{1-2r}}(\Omega)\subset L^{\frac{3\rho}{2-r}}(\Omega)$$
and Lemma \ref{tutt}
to estimate the right-hand side as
\begin{align*}
-\l|\partial_t u|^\rho\partial_{tt}u,A^{r+\sigma}\partial_t w\r
&\leq\|\partial_t u\|_{L^{3\rho/(2-r)}}^\rho\|\partial_{tt}u\|_{L^{6/(1-2\sigma)}}
\|A^{r+\sigma}\partial_t w\|_{L^{6/(1+2r+2\sigma)}}\\
&\leq c\|\partial_t u\|_{1+r}^\rho\|\partial_{tt} u\|_{1+\sigma}
\|\partial_t w\|_{1+r+\sigma}\\
&\leq c \|\partial_t w\|_{1+r+\sigma}.
\end{align*}
Besides, due to \eqref{GROW}, we have
\begin{align*}
-\l f(u),A^{r+\sigma}\partial_{t} w\r
&\leq \|f(u)\|_{L^{6/(5-2r-2\sigma)}}\|A^{r+\sigma}\partial_{t} w\|_{L^{6/(1+2r+2\sigma)}}\\
&\leq
 c\big(1+\|u\|^5_{1+r}\big)\|\partial_{t}w\|_{1+r+\sigma}\\
 &\leq c\|\partial_{t}w\|_{1+r+\sigma}.
\end{align*}
Furthermore, since $r+\sigma\leq \frac{1+r+\sigma}{2}$,
$$
\l h,A^{r+\sigma}\partial_{t} w\r
\leq \|h\|\|A^{r+\sigma}\partial_{t} w\|
\leq c\|h\|\|\partial_{t}w\|_{1+r+\sigma},
$$
showing that
$$\l \gamma, \partial_t w\r_{r+\sigma}\leq c\|\partial_{t}w\|_{1+r+\sigma}.$$
Analogous computations provides the estimate
$$\delta\l \gamma, w\r_{r+\sigma}\leq \delta c\|w\|_{1+r+\sigma}.$$
We thus end up with the differential inequality
$$
\ddt\Lambda_{r+\sigma}+\frac{\delta}{2} E_{r+\sigma}
\leq c
$$
for some $\delta>0$.
In light of~\eqref{eqis-r}, and recalling that $E_{r+\sigma}(0)=0$, from the Gronwall lemma
we infer that
\begin{equation}
\label{bound_attr2}
\|K(t)z\|_{\HH^{r+\sigma}}\leq c,
\end{equation}
for some $c=c(\A)>0$.
By virtue of \eqref{bound_attr1} and \eqref{bound_attr2},
we conclude that
$$\dist_{\HH}(\A,\BB_{r+\sigma})=\dist_{\HH}(S(t)\A,\BB_{r+\sigma})\leq c\e^{-\omega t}\to 0,$$
where $\BB_{r+\sigma}$ is a closed ball
of $\HH^{r+\sigma}$ of radius sufficiently large. In particular, this yields the inclusion
$\A\subset \BB_{r+\sigma}$
\end{proof}

\begin{proof}[Proof of Theorem \ref{attrH1}]
By reiterated applications of Lemma~\ref{keylemma},
in a finite number of steps we arrive to show that $\A$ is bounded in $\HH^{\s+\sigma}$ with $\s+\sigma=1$.
\end{proof}

\begin{proof}[Proof of Proposition \ref{oppy}]
Let $\zeta(t)=(u(t),\partial_t u(t),\eta^t)$ be a {\cbt}, that is, a solution lying on $\A$.
Fixed an arbitrary $k>0$, let us consider the solution at time $\tau>0$
with initial data $\zeta(t-k)$
$$S(\tau)\zeta(t-k)=(v(\tau),\partial_t v(\tau),\xi^\tau).$$
Observing that
$$v(\tau)=u(t-k+\tau)\and \xi^{\tau}=\eta^{t-k+\tau},$$
the representation formula~\eqref{REP} applied to $\xi^\tau$ yields
$$
\eta^{t-k+\tau}(s)=\xi^\tau(s)=v(\tau)-v(\tau-s)
=u(t-k+\tau)-u(t-k+\tau-s),
$$
for every $s\leq \tau$. Letting now $k=\tau$, we obtain \eqref{ForE} for all $s\leq\tau$,
and from the arbitrariness of $\tau>0$ the claim follows.
\end{proof}

\section*{Appendix: Some Technical Results}
\setcounter{equation}{0}
\setcounter{subsection}{0}
\renewcommand{\theequation}{A.\arabic{equation}}
\renewcommand{\thesubsection}{A.\arabic{subsection}}
\theoremstyle{plain}
\newtheorem{lemmaA}{Lemma}[section]
\renewcommand{\thelemmaA}{A.\arabic{lemmaA}}

\subsection{An auxiliary problem}
Let $\s\in[0,1]$. For a sufficiently regular function $\gamma=\gamma(t)$ on $[0,\infty)$,
let us consider the Cauchy problem in $\HH^\s$
\begin{equation}
\label{SYSaux}
\begin{cases}
\displaystyle
A\partial_{tt}u+A\partial_t u+Au+ \int_0^\infty \mu(s) A\eta(s)\,\d s =\gamma,\\
\partial_t\eta=T\eta+ \partial_t u,
\end{cases}
\end{equation}
with related energy
$$E_\s(t)=\frac12\big[\|u(t)\|_{1+\s}^2+\|\partial_t u(t)\|_{1+\s}^2+\|\eta^t\|_{\M^\s}^2\big].$$

\begin{lemmaA}
\label{Abound_SYS}
For all $\delta>0$ small, there exists $\Lambda_\s$ satisfying
\begin{equation}
\label{Aeqis}
\frac12E_\s\leq \Lambda_\s\leq 2E_\s
\end{equation}
and
\begin{equation}
\label{AeqisB}
\ddt\Lambda_\s+\delta E_\s
\leq\l \gamma,\partial_t u\r_\s+\delta\l \gamma,u\r_\s.
\end{equation}
\end{lemmaA}

\begin{proof}
We multiply \eqref{SYSaux} by $(\partial_t u,\eta)$ in $\H^{\s}\times\M^\s$, so obtaining
$$
\ddt E_\s+\|\partial_t u\|^2_{1+\s}=\l TA^{\frac{\s}2}\eta,A^{\frac{\s}2}\eta\r_{\M}+\l \gamma,\partial_t u\r_\s
\leq \l \gamma,\partial_t u\r_\s.
$$
We now define the functionals
\begin{align*}
\Psi_\s(t)&=\int_0^\infty\Big(\int_s^\infty \mu(y)\,\d y\Big)\|\eta^t(s)\|_{1+\s}^2\,\d s,\\
\Phi_\s(t)&=\frac12\|u(t)\|_{1+\s}^2+\l u(t),\partial_t u(t)\r_{1+\s},
\end{align*}
which satisfy the bounds (see \eqref{NEC1})
$$0\leq\Psi_\s\leq \Theta\|\eta\|_{\M^\s}^2$$
and
$$|\Phi_\s|\leq \|u\|_{1+\s}^2+\frac12\|\partial_t u\|^2_{1+\s}.$$
Setting
$$\Lambda_\s(t)=E_\s(t)+\eps \Psi_\s(t)+\delta\Phi_\s(t),$$
inequality \eqref{Aeqis} is easily seen to hold for every
$\eps\leq\frac{1}{2\Theta}$ and $\delta\leq\frac14$.
Taking the time derivative of $\Psi_\s$ we get
\begin{align*}
\ddt\Psi_\s+\frac12\|\eta\|_{\M^\s}^2&=-\frac12\|\eta\|_{\M^\s}^2
+2\int_0^\infty\Big(\int_s^\infty \mu(y)\,\d y\Big)\l\eta(s),\partial_t u\r_{1+\s}\,\d s\\
&\leq-\frac12\|\eta\|_{\M^\s}^2+2\Theta\sqrt{\kappa}\|\eta\|_{\M^\s}\|\partial_t u\|_{1+\s}\\
\noalign{\vskip2mm}
&\leq 2\Theta^2\kappa\|\partial_t u\|_{1+\s}^2.
\end{align*}
Besides, a multiplication of the first equation of \eqref{SYSaux} by $A^\s u$ provides
\begin{align*}
\ddt\Phi_\s+\frac12\|u\|^2_{1+\s}&=-\frac12 \|u\|^2_{1+\s}+\|\partial_t u\|^2_{1+\s}
-\int_0^\infty\mu(s)\l \eta(s),u\r_{1+\s}\,\d s
+\l \gamma,u\r_\s\\
&\leq \|\partial_t u\|^2_{1+\s}+\frac\kappa2\|\eta\|_{\M^\s}^2+\l \gamma, u\r_\s.
\end{align*}
Collecting the inequalities above, we end up with
$$
\ddt\Lambda_\s+\frac\delta2\|u\|^2_{1+\s}+(1-2\eps\Theta^2\kappa-\delta)\|\partial_t u\|^2_{1+\s}
+\frac12(\eps-\delta\kappa)\|\eta\|_{\M^\s}^2
\leq\l \gamma,\partial_t u\r_\s+\delta\l \gamma,u\r_\s.
$$
Fixing $\eps\in \big(0,\frac{1}{2\Theta}\big]$ such that
$$1-2\eps\Theta^2\kappa\geq \frac12,$$
inequality \eqref{AeqisB} holds for every $\delta>0$ small.
\end{proof}

\subsection{Two lemmas}
We finally recall two results needed in the investigation.
The first is a compactness lemma in the space $\M$ proved in~\cite{PZ}
(see Lemma 5.5 therein), while the second one is Theorem 3.1 from~\cite{WIN},
written here in a suitable form for our scopes.

\begin{lemmaA}
\label{lemmaPZ}
Let $\Xi$ be a subset of $\D(T)$, and let $r>0$.
If
$$
\sup_{\eta\in\Xi}\big[\|\eta\|_{\M^r}+\|T\eta\|_{\M}\big]<\infty
$$
and the map
$$
s\mapsto \sup_{\eta\in\Xi}\mu(s)\|\eta(s)\|_{1}^2$$
belongs to $L^1(\R^+)$, then $\Xi$ is precompact in $\M$.
\end{lemmaA}

\begin{lemmaA}
\label{teo_win}
Let $\BB\subset \HH$ be a bounded absorbing set for $S(t)$, and let $r>0$.
For every $z\in\BB$, assume there exist two operators $V_z(t)$ and $U_z(t)$
acting on $\HH$ and $\HH^{r}$, respectively, with
the following properties:
\begin{enumerate}
\item[(i)] given any $y\in\BB$ and any $x\in\HH^{r}$ satisfying the relation $y+x=z$,
$$S(t)z=V_z(t)y+U_z(t)x;$$
\item[(ii)] there exists a positive function $d_1$ vanishing at infinity
such that, for any $y\in\BB$,
$$\sup_{z\in\BB}\|V_z(t)y\|_{\HH}\leq d_1(t)\|y\|_{\HH};$$
\item[(iii)] there exists a positive function $d_2$ vanishing at infinity
such that, for any $x\in\HH^{r}$,
$$\sup_{z\in\BB}\|U_z(t)x\|_{\HH^{r}}\leq d_2(t)\|x\|_{\HH^{r}}+\Q(t),$$
for some $\Q(\cdot)$ independent of $x$.
\end{enumerate}
Then, $\BB$ is exponentially attracted by a closed ball $\BB_r$ of $\HH^{r}$
centered at zero; namely,
there exist (strictly) positive constants $c,\varkappa$ such that
$$\dist_{\HH}(S(t)\BB,\BB_r)\leq c\e^{-\varkappa t}.$$
\end{lemmaA}



\bigskip


\end{document}